\documentclass[a4paper,oneside,portrait,12pt]{amsart}

\usepackage[margin=1in]{geometry}

\usepackage[ps,all,arc,rotate]{xy}
 
\usepackage{amsfonts,amscd,amsthm,amsgen,amsmath,amssymb}
\usepackage[all]{xy}
\usepackage[vcentermath]{youngtab}
\usepackage {graphicx}
\usepackage{epstopdf}
\usepackage{amsfonts}
\usepackage{amsthm}
\usepackage{amsmath}
\usepackage{amsfonts}
\usepackage{latexsym}
\usepackage{amssymb}
\usepackage{epsfig,color}
\usepackage{graphicx, amssymb}
 \usepackage{color}
 \usepackage{epsfig,color}
\usepackage{graphicx}
\usepackage{caption}
\usepackage{subcaption}
\usepackage{wrapfig}
\usepackage {graphicx}
\usepackage{amsfonts}
\usepackage{amsthm}
\usepackage{amsmath}
\usepackage{amsfonts}
\usepackage{latexsym}
\usepackage{amssymb}
\usepackage{epsfig,color}
\usepackage{graphicx, amssymb}
\usepackage{epstopdf}

\usepackage[colorlinks,hyperindex]{hyperref}
	\hypersetup
	{
		colorlinks,%
 		citecolor=blue,%
		linkcolor=blue,%
		urlcolor=blue,%
	}\hypersetup{
  bookmarks=true,
}

\newtheorem{definition}[]{Definition} 
\newtheorem{theorem}[definition]{Theorem}

\newtheorem{proposition}[definition]{Proposition}
\newtheorem{remark}[definition]{Remark}
\newtheorem{lemma}[definition]{Lemma}
\theoremstyle{definition}
\newtheorem{example}{Example}[definition]

\newcommand{\g}{{\mathfrak g}}

\def\CM{{\mathcal M}}

\def\CO{{\mathcal O}}
\def\CP{{\mathbb {CP}}}

\def\CPP{{\mathcal P}}

\def\phi{{\psi}}

\newcommand{\Id}{\operatorname{Id}}
\newcommand{\id}{\operatorname{\text{\sf id}}}

\newcommand{\C}{\mathbb{C}}
\newcommand{\Z}{\mathbb{Z}}
\newcommand{\R}{\mathbb{R}}

\newcommand{\N}{\mathbb{N}}

\newcommand{\SL}{{\rm SL}}
\newcommand{\GL}{{\rm GL}}
\newcommand{\PSL}{{\rm PSL}}

\title{Branes through finite group actions}

\author[S]{Sebastian Heller}
\author[L,m]{Laura P. Schaposnik}

\address[S]{Department of Mathematics, University of Hamburg,  20146 Hamburg, Germany}
\address[L]{Department of Mathematics, University of Illinois at Chicago, Chicago, 60647, USA }

\date{\today}

\begin{document}

\title{Branes through finite group actions}


\maketitle
\begin{abstract}
Mid-dimensional  $(A,B,A)$ and $(B,B,B)$-branes in the moduli space of flat $G_{\C}$-connections  appearing from finite group actions on compact Riemann surfaces are studied. The geometry and topology of these spaces  is then described via the corresponding Higgs bundles and Hitchin fibrations.
\end{abstract}







\section{Introduction - mid-dimensional $(B,B,B)-branes$}\label{introduction}
 
This paper is dedicated to the study of mid-dimensional subspaces of the neutral connected component of the
 moduli space of flat $G_{\mathbb{C}}$-connections on a compact Riemann surface $\Sigma$ of genus $g\geq 2$, for $G_{\mathbb{C}}$ a complex Lie group, associated to finite group actions on $\Sigma$.  
As shown in \cite{N1,Simpson}, considering suitable stability conditions, a Higgs bundle defines a solution of equations for a $G$-connection $A$ known as the Hitchin equations, where $G$ is the maximal compact subgroup of $G_\C$. In particular, for $G = U(n)$ these are $F_A + [\Phi, \Phi^*] = 0$ and the vanishing of the antiholomorphic
part of the covariant derivative of $\Phi$, this is, $d''_A\Phi=0$.
In such case,   the connection $\nabla_A + \Phi + \Phi^*$ is flat, with holonomy in $\GL(n,\C)$.  
  We shall denote by $\mathcal{M}_{G_{\mathbb{C}}}$
  the moduli space  of $G_{\mathbb{C}}$-Higgs bundles on a compact Riemann surface $\Sigma$, the space of solutions to the Hitchin equations on the  surface   modulo gauge equivalence.


The space of solutions to Hitchin's equations is 
   a hyper-K\"ahler manifold, and thus there is a family of complex structures from which we shall fix  $I,J,K$ obeying    quaternionic equations; along the paper we shall fix those structures following the notation of \cite{slices,N1,Kap}.  With this convention, the smooth locus of $\mathcal{M}_{G_{\mathbb{C}}}$ corresponds to the space of solutions to Hitchin's equations together with  the complex structure $I$. Throughout this work we shall adopt the physicists' language in which a Lagrangian submanifold supporting a flat connection is called an {\em A-brane}, and a complex submanifold supporting a complex sheaf is a {\em B-brane}. By considering the support of branes, one may say that a submanifold of a hyper-K\"ahler manifold is of   of type $A$ or $B$ with respect to each of the   structures, and hence one may speak of branes of type $(B,B,B), (B,A,A), (A,B,A)$ and $(A,A,B)$. Since understanding the support of branes is already a difficult endeavor, throughout the paper we shall consider only the support of branes and study their appearance within the moduli space of Higgs bundles. With a mild abuse of notation, we shall refer to the support of branes as branes themselves.

   The construction and study of branes in the moduli space of $G_\C$-Higgs bundles through actions on the group $G_\C$ or on the surface $\Sigma$ only began a few years ago - see, for example, \cite{aba} for a first appearance of branes through actions on the surface, and  \cite{slices} for actions on groups. Whilst one may describe those branes inside the Hitchin fibration of $\CM_{G_\C}$ (obtaining for instance a nonabelian fibration of $(B,A,A)$ as described in \cite{nonabelian}, or a real integrable system through $(A,B,A)$-branes as in \cite{aba})  not much more is known about the geometry of those branes constructed, or how other interesting subspaces appear through group actions.  

The construction of hyperholomorphic branes, or $(B,B,B)$-branes, is of particular interest, as not many nontrivial examples have been found - one may construct $(B,B,B)$-branes by considering the moduli space of $H_\C$-Higgs bundles inside the moduli space of $G_\C$-Higgs bundles, for $H_\C$ a complex subgroup of $G_\C$, but how else may one define $(B,B,B)$-branes? 
   After overviewing finite group actions on flat connections, in Section \ref{connections1}  new $(B,B,B)$-branes are constructed:
\smallbreak

 \noindent {\bf Theorem \ref{hyperkaehlersub}}: {\it 
Let  $\Sigma$ be  a compact Riemann surface of genus $g\geq2$ and $\Gamma$  a finite group acting on $\Sigma$ by holomorphic automorphisms. The connected components of the space of gauge equivalence classes of irreducible $\Gamma$-equivariant flat $G_\C$-connections are hyper-K\"ahler submanifolds
of the moduli space of flat irreducible $G_\C$-connections on $\Sigma$, and hence give $(B,B,B)$-branes in the moduli space  of $G_{\C}$-Higgs bundles}.

\smallbreak
Since mid-dimensional spaces may be $A$-branes or $B$-branes with respect to each of the   structures, 
 it is particularly interesting to seek    finite group actions giving mid-dimensional hyper-K\"ahler submanifolds. In  Section \ref{CMC_2} we give a classification of   actions leading to mid-dimensional branes in  the moduli space $\mathcal{M}_{dR}^g$ of flat $\SL(2,\C)$-connections on a compact Riemann surface of genus $g$
  \smallbreak
  
  \noindent {\bf Theorem \ref{only1}}: {\it  Let $\Sigma$ be a compact Riemann surface of genus $g\geq2$ and $\Gamma$ be a finite group   acting on  $\Sigma$ by holomorphic automorphisms
 such that  a component of the moduli space of $\Gamma$-equivariant flat $\SL(2,\C)$-connections on $\Sigma$ has half the dimension
of the moduli space $\mathcal M^g_{dR}$.
Then one  of the following holds:}
\begin{itemize}
\item[(I)]  {\it $\Gamma=\Z_2$ acts by a fixed points free involution on $\Sigma$, or}
\item[(II)]  {\it $\Sigma$ is hyperelliptic of genus $3$ and $\Gamma=\Z_2\times \Z_2.$}
\end{itemize}
{\it In   case (II), one of the $\Z_2$-factors corresponds to the hyperelliptic involution, whilst the other $\Z_2$-factor corresponds to an involution with 4 fixed points. }

\smallbreak

  Although we highlight results which hold for generic groups, most of the work in the remaining sections is done for $G_{\mathbb{C}}=\SL(2,\C)$. In order to understand the geometry of these branes, we consider them inside the moduli space of Higgs bundles and look at their intersection with smooth fibres of the Hitchin fibration. 
  After overviewing the  $\SL(2,\C)$-Hitchin fibration in Section \ref{2cov}, 
  in Section \ref{fixer}  we obtain the following geometric description of the intersection of the $(B,B,B)$-brane of Theorem \ref{only1} (I) with the regular fibres of the Hitchin  fibration:
  \smallbreak
  
  \noindent {\bf Theorem \ref{Aabelian-variety}}:  {\it 
Let $\tau$ be a fixed point free involution. Then, the $\tau$-equivariant $(B,B,B)$-brane intersects a generic fibre  of the  Hitchin fibration over a point defining the spectral curve $S$
 in the abelian variety  $ {\rm Prym}(S/\tau,\Sigma/\tau)/\Z_2.$
 } \smallbreak
   
 In order to study  the equivariant $(B,B,B)$-branes of Theorem \ref{theorem_3} (II),  it is shown  in Section \ref{genus_3} that it suffices to consider Higgs bundles over hyperelliptic surfaces of genus 3  with fixed point free actions. In such case, we  describe the intersection of the brane with the regular fibres of the Hitchin fibration   in Theorem \ref{ultimo}.
      Finally, from \cite[Section 12]{Kap}, the equivariant and anti-equivariant spaces considered in this short paper  have dual branes in the moduli spaces of Higgs bundles for the Langlands dual group.  

      We conclude the work in Section \ref{duality} with comments on Langlands duality for the branes constructed in this paper, and noting   it is interesting to compare the spaces in Theorem \ref{hyperkaehlersub} with spaces appearing in other papers - e.g. the real integrable systems  given by $(A,B,A)$-branes  in \cite[Theorem 17]{aba} and  the $(B,A,A)$-branes of \cite{slices}.

Since the work of the present paper was first announced at the Simons Center of Geometry and Physics in June 2016 (at the conference ``New perspectives
on Higgs bundles, branes and quantization''), actions of finite groups on the moduli space of Higgs bundles and of flat connections have received increasing attention - an  interested reader in the subject might want to look among other papers at the work of Schaffhauser \cite{some2} for actions on moduli spaces of vector bundles, and of Hoskins-Shaffhauser \cite{Some3} for interesting branes arising from group actions on quiver varieties.

\subsection*{Acknowledgments}
 This paper began during the workshop {\it Higgs Bundles in Geometry and Physics}  at the University of Heidelberg February 28-March 3 2016, and the authors would like to thank the organizers for such a stimulating environment.   The work of LPS is partially supported by NSF DMS-1509693 and the work of SH is partially supported by DFG HE 6828/1-2.
 
\section{Equivariant flat connections} \label{section_finite}\label{equivariant_conncetions}


Consider $\Sigma$ a compact (connected) Riemann surface of genus $g\geq2$ and a finite group action $\Gamma\times\Sigma\to\Sigma.$ These actions have been studied by many researchers and in the case of surfaces of genus 2 and 3, a complete classification of all finite group actions is given in \cite[Tables 4, 5]{B91}. 
Moreover, in the case of actions induced on rank 2 bundles through automorphisms of $\Sigma$, a very concrete description of the fixed points in terms of parabolic structures is given in \cite{andersen}.

 In order to understand flat equivariant $G_\C$-connections on $\Sigma,$ one needs to first fix a $C^\infty$ trivialization $\underline\C^n=\Sigma\times\C^n\to\Sigma$ of the underlying vector bundle.  In what follows we shall restrict our attention to the groups $G_\C=\GL(n,\C),\, \SL(n,\C),$ and thus in the case of $\SL(n,\C)$ require the trivialization  to preserve the determinant.

 \begin{definition}\label{hola1}
 A $\Gamma$-equivariant flat connection on $\Sigma$ is a flat connection $\nabla$ on $\underline\C^n\to\Sigma$
 such that for every $\phi\in\Gamma$ there exist a $\GL(n,\C)$-gauge transformation $g_\phi\colon\Sigma\to \GL(n,\C)$ for which 
 \[\phi^*\nabla=\nabla \cdot g_\phi,\]
 and where
 $\phi\mapsto g_\phi$
is a {\em generalized group homomorphism}, i.e., satisfies $g_{\id}=\id$ and $g_{(\phi\circ\tau)}(p)=g_{\tau(p)}\circ g_{\phi}(\tau(p))$.
 \end{definition}
\begin{remark}
One should note that  Definition \ref{hola1} differs from the definition of an equivariant flat connection in \cite{BF96}, but is rather in the spirit of the definition of equivariant bundles with equivariant determinants of \cite[Definition 2.2]{andersen}. We have decided
not to fix the gauge transformations $g_\psi$ beforehand by giving a lift of the action of $G$ to the bundle. In this sense, the definition is
closer to the usual meaning of equivariance. On the other hand, Definition \ref{hola1} would seem more natural when shifting the focus to the moduli spaces
of flat connections.
\end{remark}
Denote by $\Gamma(\Sigma,G_\C)$ the space of  $G_\C$-gauge transformations on the Riemann surface $\Sigma$.  
  Note that the stabilizer subgroup  of the gauge action for an irreducible flat connection $\nabla$
 is contained in the diagonal constants,
 \begin{eqnarray}Stab_\nabla\subseteq\{g=\lambda \id \mid \lambda\in\C^*\},\end{eqnarray}
 with equality in the case of $G_\C=\GL(n,\C).$
 In the case of $G_\C=\SL(n,\C)$ one may not always be able to choose 
 the generalized group homomorphism to be of the form
\begin{eqnarray}\phi\in\Gamma \mapsto g_\phi\in \Gamma(\Sigma,\SL(n,\C)).\label{more}\end{eqnarray}
When \eqref{more} holds, the connections are  $\Gamma$-equivariant flat  $\SL(n,\C)$-connections on $\Sigma$. On the other hand, there are cases in which 
the generalized group homomorphism cannot be chosen to take values in the $\SL(n,\C)$-gauge group: an example of this is the so-called hyperelliptic descent of flat $\SL(2,\C)$-connections on a genus 2 surface considered in  \cite[$\S$2]{HL14}. As can be seen in such example, the dimensions of the two corresponding components, determined by   whether the gauges are $\SL(2,\C)$-valued or not, might be different for the same group action.

 \subsection{Finite group actions on Riemann surfaces}\label{finite_1} 
Along the paper we shall distinguish between finite group actions on Riemann surfaces with or without fixed points. 
 In either case, $\Gamma$-equivariant flat $G_\C$-connection can be studied through the quotient surface. Whilst these results might be classical,  some proofs have been included here since   sources of reference for them could not be found. We restrict to the case of $G_\C=\SL(n,\C),\GL(n,\C)$:
 
 \begin{proposition}\label{no-fix-point-conncetion}
 Let $\Gamma\times\Sigma\to\Sigma$ be a finite group action by holomorphic automorphisms without fixed points.
 Then, any $\Gamma$-equivariant flat $G_\C$-connection is gauge-equivalent to the pull-back of a flat $\GL(n,\C)$-connection on $\Sigma/\Gamma.$  \end{proposition}
 
\begin{proof}
 As the group $\Gamma$ acts without fixed points, the quotient map $\Sigma\to \Sigma/\Gamma$ gives an (unbranched) covering map.
We define an equivalence relation on $\Sigma\times\C^n$ by
$(p,v)\equiv (q,w)$
if and only if there exists $\phi\in\Gamma$ such that
$\phi(p)=q$ and $w=g_\phi(p)(v).$
Then the quotient
\[(\Sigma\times\C^n)/\Gamma\to \Sigma/\Gamma\] 
is a vector bundle equipped with a natural induced flat connection $\tilde\nabla$
which pulls back to connection which is gauge equivalent to $\nabla$ on $\underline{\C^n}\to\Sigma.$
 \end{proof} 
 For $\Gamma\times\Sigma\to\Sigma$  a finite group action  by holomorphic automorphisms with fixed points,  denote by $B\subset \Sigma/\Gamma$  the image of the fixed points, giving the  branch points of the ramified cover
 \begin{eqnarray}\pi_{\Gamma} : \Sigma \rightarrow \Sigma/\Gamma.\end{eqnarray}
 From here, the following proposition follows naturally. 
 
\begin{proposition}\label{fix-point-conncetionI}
 Let $\Gamma\times\Sigma\to\Sigma$ be a group action by holomorphic automorphisms with fixed points, and let $B$ be the image of the fixed points in $\Sigma/\Gamma$. Then, any $\Gamma$-equivariant flat $G_\C$-connection    is gauge-equivalent to the extension of the pullback of a flat $\GL(n,\C)$-connection on $\Sigma/\Gamma-B.$
\end{proposition}
\begin{remark}\label{P_on_quotient}
If
 $G_\C=\SL(n,\C)$ and the generalized group homomorphism $\phi\mapsto  g_\phi$ takes values in the $\SL(n,\C)$-gauge group, then in Propositions \ref{no-fix-point-conncetion}-\ref{fix-point-conncetionI} the connection on $\Sigma/\Gamma$, resp. on $\Sigma/\Gamma-B$, can be chosen to be a $\SL(n,\C)$-connection.
 If the generalized group homomorphism $\phi\mapsto  g_\phi$ does not take values in the $\SL(n,\C)$-gauge group, then it is more
appropriate (e.g., for counting dimensions of the   moduli spaces) to consider the induced $\mathbb P\SL(n,\C)$-connection instead of the $\GL(n,\C)$-connection constructed in Prop. \ref{no-fix-point-conncetion}-\ref{fix-point-conncetionI}. 
\end{remark}

The local monodromy of the equivariant connections around images $p\in B$ of branch points can be described by considering the stabilizer group of $p$, leading to the following: 

\begin{proposition}\label{fix-point-conncetion}
Given a $\Gamma$-equivariant $G_{\C}$-connection on $\Sigma$,  the conjugacy class of the monodromy along a small loop around a point $p\in \pi_\Gamma^{-1}(B)$ is given by a root of the identity. Moreover, for sufficiently close (with respect to the supremum norm)
 irreducible flat $\Gamma$-equivariant connections, the conjugacy classes of these local monodromies are the same.
 \end{proposition}
 
\begin{proof}
Recall that the stabilizer group $Stab_p\subset\Gamma$ of a point $p\in\Sigma$ satisfies   $Stab_p=\Z_k$ for some $k\geq1,$ and is non-trivial if and only if $p$ is a branch point of $\pi_{\Gamma}\colon\Sigma\to\Sigma/\Gamma.$ The local monodromy around $\pi_{\Gamma}(p)=:q$ can then   be determined as follows: given $\phi$  a generator of $Stab_p$, one has that $\phi^k=\id$ and $g_\phi(p)^k=\id$.   Note that if $\tau(p)=\tilde p\neq p$ for some $\tau\in\Gamma,$ then the order of the stabilizer group of
 $\tilde p$ is the same as the one of $p,$ and the conjugacy class of the corresponding gauge $g_{\tau\circ\phi\circ\tau^{-1}}(\tilde p)$ as a subset of $\GL(n,\C)$ is also the same as
 the conjugacy class of $g_\phi(p)$. Then, $g_\phi(p)$ represents the conjugacy class
of the local monodromy around $q$, and this can be seen as follows. Indeed, consider   a singular connection $\tilde\nabla$ on the punctured disc. By means of the Deligne extension procedure, the connection $\tilde\nabla$ is gauge equivalent to
\begin{equation}\label{local_conncetion_covering}
d+A\frac{dz}{z}
\end{equation}
for some $A\in\g_\C.$ Then, near $p\in\Sigma$ the connection $\nabla=d+\omega$ is gauge equivalent via a singular (but single valued) gauge transformation $g$ to the pullback of 
\eqref{local_conncetion_covering} via the covering map $z\mapsto z^k,$ and this gauge transformation is invariant under $z\mapsto e^{\tfrac{2\pi i}{k}}z.$ On the other hand, the transformation 
$g$ can be written as
$g=z^{-kA} h(z)$
where $h$ is a single valued smooth map $U\subset\Sigma\to G_\C.$ The conjugacy class of the local monodromy of \eqref{local_conncetion_covering} is given by  $z^{-A}$
which proves the claim as $g$ is invariant and  $\omega=h^{-1}dh.$

Finally, in order to see that nearby irreducible flat equivariant connections $\nabla^1$ and $\nabla^2$ give rise to the same local monodromies on the quotient surface, 
  consider gauges $g^i_\phi$ for which  $\phi^*\nabla^i=\nabla^i\cdot g^i_\phi$ and $g^i_\phi(p)^k=\id$. 
  Note that the conjugacy class of $g^i_\phi$
   at a fixed point of $\phi$ as well as the local monodromy around the corresponding branch point are also determined by the parallel transport of the original connection along a (non-closed) lift
   of the small loop around the branch point on $\Sigma/\Gamma.$ Of course, we use the background trivialization $\underline{\C^n}\to\Sigma$ of the underlying bundle with respect to the trivial connection $d$
   to do so. Then, by standard estimates for solutions of linear ODE's, it follows that equivariant connections whose connection 1-forms are close with respect to the supremum
   norm give rise the same conjugacy class (as they  are roots of the identity).
\end{proof} 



As mentioned in Section \ref{introduction},  subspaces of the smooth loci of the moduli space of $G_\C$-Higgs bundles, or equivalently,   $G_\C$-flat connections, may be holomorphic or Lagrangian with respect to some of the fixed complex structures and symplectic forms, giving what we refer to as $B$-branes and $A$-branes in those structures \cite{Kap}. In the forthcoming sections, we shall study different settings in which the space of $\Gamma$-equivariant flat irreducible $G_\C$-connections on a Riemann surface $\Sigma$ form branes.

 \subsection{Equivariant  $(B,B,B)$-branes}
\label{Branes_in_the_}\label{section_construction}\label{connections1}
In what follows we shall show  that  the moduli space of $\Gamma$-equivariant flat irreducible $G_\C$-connections on a Riemann surface $\Sigma$
is a well-defined hyper-K\"ahler submanifold of the moduli space of flat irreducible $G_\C$-connections on $\Sigma$. From the previous sections,  this space  is
a complex  submanifold of the moduli space of irreducible flat $G_\C$-connections on $\Sigma$ (with respect to the   structure $J$ induced by the complex group $G_\C$).

\begin{lemma}\label{lemma-1}
Let $\nabla$ be a $\Gamma$-equivariant flat irreducible $G_\C$-connection on $\Sigma$, and $g\colon \Sigma\to G_\C$ be a gauge transformation.
Then, $\nabla\cdot g$ is $\Gamma$-equivariant, and gives rise to the same point as $\nabla$ in the moduli space of flat (possibly singular) $G_\C$-connections on $\Sigma/\Gamma.$
\end{lemma}
 
%
%
%
%
%

The above lemma follows from  the definition of $\Gamma$-equivariant flat connections, and therefore by Propositions \ref{no-fix-point-conncetion}, \ref{fix-point-conncetionI} and   \ref{fix-point-conncetion}, the  connected components of the moduli space of $\Gamma$-equivariant flat irreducible $G_\C$-connections can be locally  identified with open subsets of the moduli space of flat irreducible connections on $\Sigma/\Gamma$ with monodromies in fixed conjugacy classes determined by the branch order of $\Sigma\to\Sigma/\Gamma$ and by the component. 
One should note that the corresponding moduli spaces of irreducible flat connections on $\Sigma$ and $\Sigma/\Gamma$ are in general not globally the same, 
as can be seen in the following example.
\begin{example}\label{example_3}
Consider a Riemann surface $\Sigma$ of genus $3$ which admits a fixed point free involution $\phi\colon\Sigma\to\Sigma$  giving a double covering to a Riemann surface\linebreak $M:=\Sigma/\Z_2$ of genus $2.$
Let $\nabla$ be a flat unitary line bundle connection such that $\nabla$ is not self-dual, i.e., $\nabla^*$ is not gauge equivalent to $\nabla,$ and
such that $\phi^*\nabla=\nabla^*$. 
Then, $\nabla\oplus\nabla^*$ is a $\Z_2$-equivariant flat  connection,
whose corresponding flat connection on $\Sigma/\Z_2$ is irreducible. A more extensive analysis of this set up is given in Section \ref{genus_3} to illustrate the results of the paper. 
\end{example}

As mentioned previously, by Hitchin's work \cite{N1} and generalizations thereof, the space of irreducible flat $G_\C$-connections is a hyper-K\"ahler manifold. 
 In this context, one has the following results.

 \begin{theorem}\label{hyperkaehlersub}
Let $\Sigma$ be a compact Riemann surface of genus $g\geq2$ and $\Gamma$ be a finite group acting on $\Sigma$ by holomorphic automorphisms. 
Then, the  connected components of the space of gauge equivalence classes of irreducible $\Gamma$-equivariant flat $G_\C$-connections are hyper-K\"ahler submanifolds
of the moduli space of flat irreducible $G_\C$-connections on $\Sigma,$ for  $G_\C$ a semi-simple Lie group. 
\end{theorem}

\begin{proof}
For  $G_\C$ a semi-simple Lie group, by Lemma \ref{lemma-1} any connected component of the space of gauge equivalence classes of irreducible $\Gamma$-equivariant flat $G_\C$-connections is locally
identified with an open subspace of the moduli space of irreducible (possibly singular at the branch points) connections on $\Sigma/\Gamma.$ 
This identification is given by taking the pull-back of the connection and, if necessary,
applying a singular gauge transformation which gauges the singularities of the pull-back connection at the fixed points of $\Gamma$  away. We call the later gauge transformation a desingularization.
We claim that taking the pull-back and desingularization
gives rise to  a (complex analytic) immersion
from the smooth part of the moduli space  of flat (possibly singular) connections on $\Sigma/\Gamma$ to the smooth part of the moduli space
of flat connections on the Riemann surface 
$\Sigma$. 

Since the complex structure $J$ is respected by the pull-back and desingularization,  it remains to prove that
the differential of taking the pull-back and desingularization is injective at every irreducible gauge class which pulls back to an irreducible gauge class.
In order to do so we note that we can work with singular connections $\nabla$ on $\Sigma/\Gamma$  which are of a local standard form $d+A\tfrac{dz}{z}$
as in \eqref{local_conncetion_covering} at any branch point. 
 Let $X\in\Omega^1(\Sigma/\Gamma,\mathfrak{g}_\C)$ represent  a tangent vector which can be chosen to vanish in an open neighborhood of the singular points of $\nabla.$  
 If $X$ is not trivial, then there exists a second tangent vector $Y\in\Omega^1(\Sigma/\Gamma,\mathfrak{g}_\C)$ (satisfying $d^\nabla Y=0$ as well), which also vanishes in an open neighborhood of the singular points of $\nabla$,
 such that
for the natural symplectic structure $\Omega$ on the moduli space of 
flat connections with fixed local conjugacy classes (as defined in \cite{AlMa})
\[0\neq\Omega_{\Sigma/\Gamma}(X,Y)=\int_{\Sigma/\Gamma} \text{tr}(X\wedge Y),\]
where $\text{tr}$ denotes the Killing form on $\mathfrak{g}_\C.$
Let us write $\pi$ for $\pi_\Gamma$, this is, for the covering $\pi\colon \Sigma\to\Sigma/\Gamma$.
 Then there is a singular gauge transformation $g$ such that
 $\pi^*\nabla.g$ is smooth on $\Sigma.$ The tangent vectors $X$ and $Y$ are then mapped to the smooth 1-forms 
 $g^{-1}\pi^*X g$  and $g^{-1}\pi^*Y g$ representing the corresponding tangent vectors in the moduli space of flat connections on $\Sigma.$
We obtain
\[\Omega_{\Sigma}(g^{-1}\pi^*Xg,g^{-1}\pi^*Yg)=\int_\Sigma\text{tr}(g^{-1}\pi^*Xg\wedge g^{-1}\pi^* Yg)=\int_\Sigma\text{tr}(\pi^*X\wedge \pi^* Y)\neq0.\]
Therefore, $g^{-1}\pi^*X g$  does not vanish in the first cohomology of $d^\nabla$ and represents a non-trivial tangent vector.
Hence, pull-back and desingularization is an immersion and therefore the moduli space of irreducible $\Gamma$-equivariant flat $G_\C$-connections is
a complex submanifold
with respect to the complex structure $J$.

In order to show that it is also a complex submanifold with respect to the complex structure $I$ 
 one needs to make use of the
uniqueness of solutions of the self-duality equations.  Fixing a trivialization of the underlying $C^\infty$ bundle in order to work on $\underline{\C^n}\to\Sigma,$ consider the 
standard hermitian metric $h$ on  $\underline{\C^n}\to\Sigma$ which is invariant under $\Gamma.$
Consider    an  irreducible $\Gamma$-equivariant flat $G_\C$-connection $\nabla$  such that $(\nabla,h)$ satisfies the self-duality equations, i.e., so that 
$\nabla=\nabla^u+\Phi+\Phi^*$
where $\nabla^u$ is unitary with respect to $h$ and $\Phi^*$ is the adjoint of the holomorphic Higgs field $\Phi$ (with respect to $h$). 
Now, for $\phi\in\Gamma$ one has that $\phi^*\nabla=\nabla  \cdot g_\phi$ decomposes
into harmonic parts as
$\phi^*\nabla=\phi^*\nabla^u+\phi^*\Phi+\phi^*\Phi^*,$
since $h$ is $\Gamma$-invariant.
Then, from   the uniqueness of solutions of the self-duality equations for stable pairs/irreducible connections (e.g. see \cite{N1} in the case of $G_\C=\SL(2,\C)$), 
the gauge transformation $g_\phi$ must already be unitary. Therefore, since  $\nabla$ is $\Gamma$-equivariant,  the corresponding Higgs pair $((\nabla^u)^{(0,1)},\Phi)$ is also $\Gamma$-equivariant (in the same sense as defined for connections), which proves the theorem.
\end{proof} 

\begin{remark}
From \cite[Corollary 3.9]{Steer95}, given any orbifold Riemann surface $\tilde\Sigma$ with negative Euler characteristic, there exists a smooth  compact Riemann surface $\Sigma$ with an action of a finite group $\Gamma$, such that $\tilde\Sigma=\Sigma/\Gamma$, and thus the analysis done in this paper could be translated in terms of Higgs bundles on orbifold Riemann surfaces. Moreover, through \cite{Steer95} one also knows that the spaces of Theorem \ref{hyperkaehlersub} are non-empty. See also the work of Anderson and Grove on equivariant bundles under group actions in \cite{andersen} for further analysis of equivariant bundles and their correspondence with parabolic bundles, and non-emptiness of these spaces. \label{orbi} \end{remark}

 From Remark \ref{orbi}, properties of the $(B,B,B)$-branes constructed through finite group actions can be deduced from \cite{Steer95}. From Theorem \ref{hyperkaehlersub} and \cite[Section 3D]{Steer95} one can describe the spaces of $\Gamma$-equivariant flat $G_{\C}$-connections as $(B,B,B)$-branes also in the case of actions with no fixed points. 
Indeed, consider $G_\C=\SL(n,\C),\GL(n,\C)$: 
in these cases  the spaces of gauge equivalence classes of irreducible $\Gamma$-equivariant flat $G_\C$-connections are   hyper-K\"ahler submanifolds
of the moduli spaces of flat irreducible $G_\C$-connections on $\Sigma,$ and are naturally covered by an open dense subset of the hyper-K\"ahler moduli space of irreducible flat $G_\C$-connections (or possibly $\PSL(n,\C)$-connections, as in  Remark \ref{P_on_quotient}) on the surface $\Sigma/\Gamma$.


 \subsection{Equivariant $(A,B,A)$-branes}\label{aba}
 
By considering a real structure $f:\Sigma\rightarrow\Sigma$ on the Riemann surface, it was shown in \cite{aba, slices} how to construct and study families of $(A,B,A)$-branes in the moduli space of $G_{\mathbb{C}}$-Higgs bundles. In particular, for $\xi$ the compact anti-holomorphic involution of $G_{\mathbb{C}}$, the fixed point set of  
$$
i_2(\bar \partial_A, \Phi):=(f^*(  \partial_A),f^*( \Phi^*  ))= (f^*(\xi(\bar \partial_A)), -f^*( \xi(\Phi) )),
$$
defines an $(A,B,A)$-brane which lies in the Hitchin fibration for $G_{\C}$-Higgs bundles $(\bar \partial_A, \Phi)$ as a real integrable system \cite[Theorem 17]{aba}. 
Moreover, for $\Sigma$ a hyperelliptic curve of genus 3, from \cite[Section 6]{GH81} and \cite[Appendix A]{aba} one can deduce further characteristics of the brane. Therefore, these Riemann surfaces shall be taken as toy models along this paper.

Following Section \ref{Branes_in_the_}, one may  consider both an orientation reversing involution with fixed points, and a finite group action on the Riemann surface $\Sigma$. By taking both involutions together, and looking at the induced action on the moduli space of flat connections, one would obtain the intersection of a $(B,B,B)$-brane and an $(A,B,A)$-brane. The compatibility and classification of these involutions is done in \cite{finite10}, the $(B,B,B)$-branes would be as described in Section \ref{section_construction}, and the $(A,B,A)$-branes as described in \cite{aba}. Therefore, by considering these two actions, one obtains natural mid-dimensional $(A,B,A)$-branes inside the $(B,B,B)$-brane, which  shall be referred to as CMC-branes.
 Finally, in \cite[Section 8]{BF96} the equivariant cohomology for equivariant bundles is calculated in terms of the action of $\Gamma$ on the usual cohomology, providing the tools to understand the cohomology of the   CMC-branes.

\section{Mid-dimensional equivariant branes}\label{CMC_2}
 
Mid-dimensional submanifolds  of the moduli space $\mathcal{M}_{G_{\mathbb{C}}}$ of $G_{\mathbb{C}}$-Higgs bundles appear to be of particular interest, as these can be both $B$-branes and $A$-branes (see, for example, the families constructed in \cite{aba,slices}). Thus, in what follows    mid-dimensional hyper-K\"ahler submanifolds coming from $\Gamma$ equivariant flat $G_\C$-connections shall be described, and a study of different finite group actions on  $\Sigma$ which lead to mid-dimensional equivariant branes shall be carried through. 

 Whilst some of the results presented in this paper can be deduced for higher rank groups, the present manuscript shall focus on the group $G_\C=\SL(2,\C)$. We denote by $\mathcal M^g_{dR}$ the moduli space of flat $\SL(2,\C)$-connections on a compact Riemann surface of genus $g$, and for simplicity drop the group label, but maintain the label for the genus since it will become of use at several stages of our analysis. 
Recall that the real dimension of this space is 
$\dim\mathcal M^g_{dR}=6g-6,$ 
and the dimension of  the moduli space $\mathcal M^{\gamma,n}_{dR}$ of flat $\SL(2,\C)$-connections on an $n$-punctured Riemann surface of genus $\gamma$ 
with fixed local monodromies (with simple eigenvalues) is
\begin{equation}\label{dim2}
\dim \mathcal M^{\gamma,n}_{dR}=6\gamma-6+2n,
\end{equation}
 provided that either $\gamma\geq 2$, or that $\gamma\geq1$  and $n\geq1$,  or finally that  $\gamma=0$ and $n\geq4$. Otherwise, the dimension of $\mathcal M^{\gamma,n}_{dR}$  is either $0$ or $2$. Note that the corresponding moduli spaces of flat $\PSL(n,\C)$-connections have the same dimensions
 as their $\SL(n,\C)$ counterparts.
In order to study mid-dimensional subspaces of $\mathcal M^g_{dR}$ coming from finite group actions, note that  the dimension of the space of $\Gamma$-equivariant $\SL(2,\C)$-connections and the order of the group $\Gamma$ are closely related. 


\begin{proposition}\label{only2}
Let $\Sigma$ be a compact Riemann surface of genus $g\geq2$, and $\Gamma$ be a  finite group of order $h$ acting on  $\Sigma$ by holomorphic automorphisms.
  If a component of the moduli space of equivariant flat $\SL(2,\C)$-connections on $\Sigma$ has half the dimension
of the moduli space $\mathcal M^g_{dR}$, then $h=2^k$ for some $k\in\N.$
\end{proposition}

\begin{proof}
In order to show that all prime factors of $h$ are $2$, note that for each prime factor $p$ of $h$ there exists a subgroup $\Gamma_p$ of $\Gamma$ such that
$\Gamma_p\cong\Z_p.$
Considering the induced group action
$\Z_p\times\Sigma\to\Sigma,$  by Riemann-Hurwitz the genera $g$ and $\gamma$ of the Riemann surfaces $\Sigma$ and $\Sigma/\Z_p$, respectively, satisfy 
\begin{equation}\label{bf}
 n(p-1)=2(g-1+p(1-\gamma)),
\end{equation}
where, as in Proposition \ref{fix-point-conncetionI},  $n=|B|$. Thus since $\dim \mathcal{M}^g_{dR}=6g-6$ one has that
\begin{equation}\label{dim1.1}
\dim \mathcal M^g_{dR}=3 n(p-1)+6 p(\gamma-1).
\end{equation}
From Section \ref{section_finite},  any component of the moduli space of flat $\Gamma$-equivariant $\SL(2,\C)$-connections on $\Sigma$ can be identified with an open subspace of the moduli space of flat $\SL(2,\C)$ or  $\PSL(2,\C)$-connections on a $n$-punctured compact Riemann surface of genus $\gamma$, for $n\in \N$ satisfying \eqref{bf}, and hence 
\begin{equation}
\label{gngamma}
3 n(p-1)+6 p(\gamma-1)\leq 12(\gamma-1)+4n.
\end{equation}
Moreover, the inequality is also valid in the special case of invariant flat connections, where one has to consider $\mathcal M_{dR}^{\gamma}$ instead of $\mathcal M^{\gamma,n}_{dR}.$
One should note that the exceptional cases in \eqref{dim2}, i.e. the case of $\gamma=0$ and $n\leq3$, and of   $\gamma=1$ and $n=0$, are excluded by dimensional reasons: indeed, in these cases $\dim \mathcal M^g_{dR}\geq6> 4\geq 2\mathcal M^{\gamma,n}_{dR}$.
 
When the genus is $\gamma>1$, the inequality \eqref{gngamma}   only holds if $p=2$. Similarly, when the genus is $\gamma=1$, equation  \eqref{gngamma} holds only when  $p=2$ (since $\gamma=1$ and $n=0$ imply that $g=1$ which would contradict the assumption that $g\geq 2$). Finally, when $\gamma=0$ the inequality  \eqref{gngamma} is equivalent to
\begin{equation}
\label{gngamma2}
(3p-7)n\leq 6(p-2),
\end{equation}
which holds only when $p=2$,  since when $n=2$ the Riemann surface $\Sigma$ would have genus $0$ which would contradict again the assumption that $g\geq 2$.
\end{proof}

\begin{theorem}\label{only1}\label{theorem_3}
Let $\Sigma$ be a compact Riemann surface of genus $g\geq2$ and $\Gamma$ be a finite group   acting on  $\Sigma$ by holomorphic automorphisms
 such that  a component of the moduli space of $\Gamma$-equivariant flat $\SL(2,\C)$-connections on $\Sigma$ has half the dimension
of the moduli space $\mathcal M^g_{dR}$.
Then one of the following holds:
\begin{itemize}
\item[(I)]  $\Gamma=\Z_2$ acts by a fixed points free involution $\tau$ on $\Sigma$, or
\item[(II)]  $\Sigma$ is hyperelliptic of genus $3$ and $\Gamma=\Z_2\times \Z_2.$
\end{itemize}
In the later case (II), one of the $\Z_2$-factors corresponds to the hyperelliptic involution $\psi$, whilst the other $\Z_2$-factor corresponds to an involution $\rho$ with 4 fixed points. 
\end{theorem}

\begin{proof} By  Proposition \ref{only2}, the order of $\Gamma$ must be $2^k$ for some $k\in\N.$ 
 When  $k=1 $, equality in \eqref{gngamma} implies that the number $n$ of fixed points is $0$ and one recovers the case (I).
Consider then $k\geq2$, in which case the finite group $\Gamma$ contains a subgroup of order $4$  isomorphic to $\Z_4$ or $\Z_2\times\Z_2.$ 
Indeed, let us assume that there does not exists any element of order $d>2$. Then, any two different elements $a$ and $b$ of order $2$  must commute since the order of their product
$ab$  is 2 by assumption. If there exists an element $a$ of order $d>2$, then $d=2^l$ for some $l\in\{2,..,k\}$ and a suitable power of $a$ generates a subgroup of $\Gamma$ which is isomorphic to $\Z_4.$
But if $\Z_4$ is a subgroup of $\Gamma,$ similar 
arguments as in the proof of  Proposition \ref{only2} apply. In fact, if we denote by $n_1$ the number of branch values 
of $\Sigma\mapsto\Sigma/\mathbb{Z}_4$ with
exactly one preimage and by $n_2$ the number of branch values with two preimages, we would obtain instead of
\eqref{gngamma} the inequality
\begin{equation}\label{9gngamma}
9 n_1+6 n_2+24(\gamma-1)\leq 12(\gamma-1)+4(n_1+n_2),\end{equation}
which only leaves the case $\gamma=0$. 
From Riemann-Hurwitz we get that $n_1$ must be even, and as $\gamma=0$ we have $n_1>0$ (as otherwise the Galois group would not contain an element of order 4).
Therefore, \eqref{9gngamma} shows that only $n_1=2$ and $n_2=1$ is possible, which is
excluded as it corresponds to $g=1.$
Hence, the group $\Gamma$ must contain $\Z_2\times\Z_2$ as a subgroup.

In what follows the two generators of the $\Z_2$-actions on $\Sigma$ shall be denoted by $\psi $ and $\rho.$
 Since the stabilizer subgroup of a point  in the Riemann surface $\Sigma$ is cyclic, the fixed points of $\psi$ and $\rho$ must be distinct. Consider then 
\[ \Sigma_{\psi}:=\Sigma/\psi=\Sigma/\Z_2;~{~\rm ~with~}~ g_\psi:=\text{genus of }\Sigma_{\psi}, \]
and   $n_{\psi}$   the number of fixed points of $\psi$. Then, by 
Riemann-Hurwitz,  one has that   $n_{\psi}=2g+2-4g_\psi,$
and therefore  the genus $g_{\rho,\psi}$ of the quotient  $\Sigma_\psi/\Z_2=\Sigma/(\Z_2\times\Z_2)$ is given by
\begin{equation}\label{mugenus}
g_{\rho,\psi}=\tfrac{1}{8}(6+2g-n_{\rho}-n_{\psi}),
\end{equation}
where $n_{\rho}$ is the number of fixed points of $\rho$ (acting on $\Sigma$).
Thus, as in Proposition \ref{only2},  since $g_{\rho,\psi}\geq0$ and $n_{\rho}, n_{\psi}\geq 0$, one must have
\begin{equation}\label{dim2.2}
6+2g\geq n_{\rho} + n_{\psi}\geq 6g-6.\end{equation}
From  equations \eqref{mugenus}-\eqref{dim2.2} one therefore  has that 
$2\leq g\leq 3.$ Note that by Theorem \ref{hyperkaehlersub} the real dimension of the moduli space needs to be divisible by 8, and thus the genus $g$ must be $g=3$. In this case,   equations \eqref{mugenus}-\eqref{dim2.2} imply that
$n_{\rho} + n_{\psi}=12.$ Moreover, as involutions on a surface of genus $3$ can not have 12 fixed points
by   Riemann-Hurwitz, it follows  that either the involutions have  $n_{\psi}=8$ and $n_{\rho}=4$ fixed points, or $n_{\psi}=4$ and $n_{\rho}=8$. Therefore,  $\Sigma$ must be hyperelliptic and, without  loss of generality, one may consider 
$\psi$ to be the hyperelliptic involution and $\rho$  an involution with 4 fixed points. 
Finally, $\Z_2\times\Z_2$ must be equal to $\Gamma$ since otherwise the dimension of the moduli space of equivariant $\SL(2,\C)$-connections would be strictly less than 6, which is a contradiction.
\end{proof} 

 \begin{remark}\label{Pg3}
The connections on the quotient space of Theorem  \ref{only1} (II) must have local monodromies
conjugate to $diag(1,-1)$ around any of the 6 branch values in $\CP^1$, and  thus are $\PSL(2,\C)$-connections, since 
otherwise the dimension of the corresponding moduli space would be strictly less than 6.
\end{remark}

One should note that a genus 3 curve $\Sigma$ is either hyperelliptic or it is a non-singular plane
quartic.  In the hyperelliptic case $\Sigma$ has exactly 8 Weierstrass points which
are the ramification points of the canonical map $\Sigma\rightarrow\CP^1$, and the action of $Aut(\Sigma)$ on the Weierstrass points can be found in \cite[Table 3]{Mag02}. More information about these actions can be found in \cite{babu} and references therein.  Moreover,   $\#Aut(\Sigma)\leq 168$, and thus only   2, 3 and 7 can divide the order of $Aut(\Sigma)$.

In the case of Theorem \ref{only1} (II) the quotient $\Sigma/(\Z_2\times\Z_2)$ is the sphere with 6 marked points: a pair of points   corresponds to the 4 fixed points of $\rho$ while the remaining 4 marked points are the image of the 8 Weierstrass points of $\Sigma.$ The space of surfaces $\Sigma$ as in Theorem \ref{only1} (II) is complex 3-dimensional, and local coordinates are given by 3 
pairwise distinct points on $\CP^1\setminus\{0,1,\infty\}.$ The space of holomorphic quadratic differentials on $\Sigma$ which are invariant under $\Gamma=\Z_2\times\Z_2$
is (naturally isomorphic to) the space of meromorphic quadratic differentials on $\CP^1$ with at most simple poles at the 6 marked points, which is also complex 3-dimensional. One should note that Example \ref{example_3} fits in this case, and further study of this setting in terms of Higgs bundles shall be given in Section \ref{genus_3}. 



\section{Higgs bundles and the Hitchin fibration} \label{equiv_bundles}\label{g3identy}
\label{2cov} 

In order to understand the geometry and topology of the mid-dimensional branes constructed in Section \ref{section_construction}, these branes shall be studied through Higgs bundles. Recall that  Higgs bundles  appeared in N.~Hitchin's work as solutions of Yang-Mills self-duality equations on a Riemann surface \cite{N1}. 
Classically, a {\it Higgs bundle} on a compact Riemann surface $\Sigma$ of genus $g\geq 2$ is a pair $(E,\Phi)$ where 
 $E$ is a holomorphic vector bundle  on $\Sigma$, 
and  $\Phi$, the {\it  Higgs field}, is a holomorphic 1-form in $H^{0}(\Sigma, {\rm End}_{0}(E)\otimes K_{\Sigma}),$
 for $K_\Sigma$ the cotangent bundle of $\Sigma$ and ${\rm End}_{0}(E)$ the traceless endomorphisms of $E$.  
This definition is for $SL(n,\C)$-Higgs bundles, and one can further generalize it to define Higgs bundles for arbitrary complex groups $G_{\mathbb{C}}$. Moreover, through stability conditions, one can construct their moduli spaces $\mathcal{M}_{G_\mathbb{C}}$. 

A natural way of studying the moduli space of Higgs bundles   is through the {\it Hitchin fibration}, 
  sending the class of a Higgs bundle $(E,\Phi)$ to the coefficients of the characteristic polynomial $\det(x I -\Phi )$. The generic fibre  is an abelian variety, which can be seen as the Jacobian variety   (or subvarieties of the Jacobian) of an algebraic curve $S$, the {\it spectral curve} associated to the Higgs field \cite{N2}. For instance in the case of classical ${\rm GL}(n,\C)$-Higgs bundles, the Hitchin base is $\bigoplus_{i=1}^{n}H^0(\Sigma,K^i_\Sigma)$ and the smooth fibres can be seen through {\it spectral data} as Jacobian varieties ${\rm Jac}(S)$ of $S$. In the case of ${\rm SL}(n,\C)$-Higgs bundles the Hitchin base is $\bigoplus_{i=2}^{n}H^0(\Sigma,K^i_\Sigma)$ and the  generic fibres are given by ${\rm Prym}(S,\Sigma)\subset {\rm Jac}(S)$.

  One can understand Higgs bundles fixed by involutions by studying the induced action on the Hitchin fibration
    (see, for instance \cite{NLie} for Higgs bundles for split real forms, and more generally \cite{thesis} for any real form,  and \cite{aba,slices} for other involutions).
 In what follows  the induced action of the finite group $\Gamma$ from Theorem \ref{hyperkaehlersub} on Higgs bundles shall be considered first, in order to later describe how the $(B,B,B)$-branes from Theorem \ref{hyperkaehlersub} intersect the fibres of the Hitchin fibration. Then,  through the duality of abelian varieties in the fibres of the Hitchin fibration, we shall comment on the dual $(B,A,A)$-branes in Section \ref{duality}.
 Since the $(B,B,B)$-branes from Theorem \ref{hyperkaehlersub} appearing as mid-dimensional spaces in the moduli space of $SL(2,\C)$-Higgs bundles correspond to Higgs bundles whose spectral curves are double covers of $\Sigma$,  some basic facts about unbranched and branched double covers of a Riemann surface $\Sigma$ shall be mentioned next.
 


Unbranched double covers  are well-known to be parametrized by $H^1(\Sigma,\Z_2)$ (e.g.~\cite{Mir}), 2-torsion   line bundles $P_2$ on $\Sigma$.   Then, for any $\alpha\in H^1(\Sigma,\Z_2)$ there exists a unique (up to gauge equivalence) flat connection $\nabla$
such that its monodromy 
$m_\nabla\colon \pi_1(\Sigma)\to H_1(\Sigma,\Z)\to\C^*$,
which is abelian, is given by $\alpha$. 
The parallel transport along (not necessarily closed) curves $\gamma$ from $q\in \Sigma$ to $q'\in \Sigma$ shall be denoted by
$ f_{\gamma}\colon (P_2)_q\to (P_2)_{q'}.$
Fixing a point $s_0$ in the fibre $ (P_2)_{q_0}\setminus\{0\}$,  for some  $q_0\in \Sigma$,  the  double cover corresponding to $\nabla$ is  \begin{eqnarray}S^\alpha:=\{ s_q\in (P_2)_q\mid q\in \Sigma;\; \exists \gamma \text{ from } q_0 \text{ to } q \text{ with } f_{\gamma}(s_0)=s_q\}\end{eqnarray} where the covering map 
$S^\alpha  \rightarrow \Sigma$ is defined by
$ s_q\mapsto q.\nonumber$ 
Note that in particular $\alpha$ is trivial if and only if $S^\alpha$ is not connected.

Branched double covers can be constructed through holomorphic sections $s\in H^0(\Sigma,L^2)$, for $L$ a holomorphic line bundle on $\Sigma$. For simplicity, one may   restrict to sections which have simple zeros\footnote{This condition will become apparent and natural when considering the Hitchin fibration, since the smooth locus is given by sections in $H^0(\Sigma,K^2)$ with simple zeros, i.e., the locus of points in the base corresponding to smooth spectral curves.}. Then, there is a (unique) double cover 
$\pi\colon S\to \Sigma$ branched over the zeros of $s$ 
such that a square root
$t\in H^0( S,\pi^*L)$ satisfying
$t^2=s$
exists. 
The cover is then given by
\begin{eqnarray} S=\{t_q\in L_q\mid t_q^2=s_q\}.\label{hat}\end{eqnarray}
 Two double covers
$\pi\colon  S\to \Sigma$
and
$\tilde\pi\colon \tilde S\to \Sigma$ are said to
differ by the flat (holomorphic) $\Z_2$-bundle $P_2$
if they correspond to the same holomorphic section $s\in H^0(\Sigma,L^2)$,  but are obtained through  the line bundles $L$ and $P_2\otimes L$ respectively.

\section{Equivariant branes and fixed point free actions}\label{fixer}
As seen in Theorem \ref{theorem_3} (I), mid dimensional equivariant $(B,B,B)$ branes in the moduli space of Higgs bundles can be constructed through fixed point free actions on the Riemann surface $\Sigma$. Thus, in what follows,  these induced branes shall be seen inside the $SL(2,\C)$ Hitchin fibration. 
 \subsection{Fixed point free actions and equivariant Higgs bundles}\label{equiv_data2}
Let $\Sigma$ be a Riemann surface  of genus $g=2\gamma-1$ with a fixed point free involution $\tau,$ and $\Sigma_\tau:=\Sigma/\tau$ its quotient, which has genus $\gamma.$
Every $\tau$-invariant holomorphic quadratic differential $Q$ on $\Sigma$ is the pull-back of a holomorphic quadratic differential $Q_\tau$ on $\Sigma_\tau,$ and as done previously, we shall restrict to those with only  simple zeros (i.e., to generic points in the Hitchin base). Consider $S$ and $S_\tau$, the double covers of $\Sigma$ and $\Sigma_\tau$ defined by $Q$ and $Q_\tau$, respectively.
The involution $\tau$ lifts to a fixed point free involution on $S$, denoted by the same symbol, and   $S/\tau=S_\tau.$ Moreover,  the involution $\tau$ and the involution $\sigma$ switching the sheets of the double cover $S$ commute.
\begin{remark}\label{invariantsc}
Recall that for a $\tau$-equivariant $\SL(2,\C)$-connection $\nabla$, the corresponding
connection on $\Sigma_\tau$ is either a $\SL(2,\C)$-connection or a $\PSL(2,\C)$-connection.
We call connections of the first type invariant, and note that connections of the latter type  can not be realized by an  invariant
trace-free connection 1-form. Therefore, the $(B,B,B)$-brane has two (connected) components, of which  the first one shall be referred to
as the $\tau$-invariant $(B,B,B)$-brane.
\end{remark}
\begin{theorem}\label{Aabelian-variety}
Let $\tau$ be a fixed point free involution as in Theorem \ref{theorem_3}, and $q$ a generic point in the $SL(2,\C)$ Hitchin base, corresponding to a spectral curve $S$. The intersection of the $\tau$-equivariant $(B,B,B)$-brane with the generic fibre over $q$
 is the abelian variety   $$ \CPP:=\{L\in Jac(S)\mid \sigma^*L=L^*,\tau^*L=L\}.$$ This space admits an identification  with 
 \begin{eqnarray} \CPP={\rm Prym}(S_\tau,\Sigma_\tau)/\Z_2\subset {\rm Prym}(S,\Sigma),\label{sym-prym}\end{eqnarray}
where the generator of $\Z_2$ is the holomorphic line bundle $P_2$ on $S_\tau$ corresponding to the (unbranched) covering $S\to S_\tau$.
\end{theorem}
\begin{proof}
From \cite{N1}, the $SL(2,\C)$ Hitchin base is given by $H^{0}(\Sigma, K_\Sigma^2)$, and a generic fibre of the Hitchin fibration over  $Q\in H^{0}(\Sigma, K_\Sigma^2)$ with simple zeros is given by  the   Prym variety 
\begin{eqnarray}\{L\in {\rm Jac}(S)\mid \sigma^*L\otimes L=\CO\}.\end{eqnarray}
 If a Higgs pair $(E,\Phi)$ is invariant with respect to the automorphism $\tau$, then 
the corresponding eigenline bundle $L$ must be symmetric, i.e., $\tau^*L=L.$
%
%
In order to show that    one can identify $\CPP$,    a subvariety of ${\rm Jac}(S)$, with the quotient ${\rm Prym}(S_\tau,\Sigma_\tau)/\Z_2$, 
 note that via pull-back one has a surjective map
\begin{eqnarray}{\rm Prym}(S_\tau,\Sigma_\tau)\to \CPP\subset {\rm Prym}(S,\Sigma).\label{map1}\end{eqnarray}
Surjectivity of the above map can be seen from looking at the dimensions of the corresponding abelian varieties, which are the same by construction.
It remains to compute the kernel of the map \eqref{map1}. Let $L$ be a holomorphic line bundle of degree 0 on $S_\tau$ which pulls back to the trivial holomorphic line bundle on
$S,$ and equip the line bundle with its unique compatible unitary flat connection $\nabla$.
On $S,$ this flat connection is trivial. Thus $\nabla$ has monodromy $\pm 1$ along closed curves on $S_\tau,$ and $-1$ is only possible when a lift of a loop
to $S$ does not close. Therefore, the line bundle $L$ is either the trivial bundle, or the bundle $P_2$ corresponding to the covering $S\to S_\tau$.
\end{proof} 

 \subsection{Fixed point free actions and anti-equivariant Higgs bundles}\label{equiv_data2b}
Let $\Sigma$ be a 
   Riemann surface  of odd genus $g=2\gamma-1$ with a fixed point free involution $\tau$ as in Theorem \ref{theorem_3}.
Following  Section \ref{CMC_2}, {\it anti-equivariant Higgs bundles $(\bar\partial,\Phi)$} with respect to the involution $\tau$ are given by a $\tau$-equivariant holomorphic structure $\bar\partial$ 
and a Higgs field $\Phi$
which satisfies \[\tau^*\Phi=-g^{-1}\circ \Phi\circ g,\]
where $g$ is the gauge transformation such that $\tau^*\bar\partial=\bar\partial.g$
and $(g\circ\tau) g=\id.$ 
In this situation,   the following analog of Theorem \ref{hyperkaehlersub} can be proved:
\begin{proposition}\label{anti-equivariant-BAA}
The connected components of the moduli space of stable anti-equivariant Higgs bundles are 
complex submanifolds of the moduli space of Higgs bundles with respect to $I$ and Lagrangian with respect to the symplectic forms corresponding to $J$ and $K$.
\end{proposition}
\begin{proof}
Over the locus of regularly stable bundles, the full moduli space is the cotangent bundle and the complex structure decouples. The moduli space of 
equivariant stable bundles is a complex submanifold of the moduli space of stable bundles, and anti-equivariant Higgs fields for a $\tau-$equivariant 
bundle are a complex subvector space of the space of Higgs fields. Hence, gauge classes of anti-equivariant Higgs bundles give rise to a complex submanifold.
To see that the holomorphic symplectic form vanishes on this submanifold, we just apply the transformation formula for
the diffeomorphism $\tau$ to the integral defining the holomorphic symplectic structure.

The dimension of a connected component of the moduli space of stable anti-equivariant Higgs bundles is  half of the dimension of the full moduli space:
because the moduli space of equivariant Higgs bundles is a hyper-K\"ahler submanifold at its smooth points, the moduli space of equivariant Higgs bundles with zero Higgs field has at its smooth points $\tfrac{1}{4}=(\tfrac{1}{2})^2$ the dimension of the full moduli space. Moreover,  for a given stable equivariant Higgs pair $(\bar\partial,0)$ with zero Higgs field, the dimension of the vector space of equivariant Higgs fields with respect to $\bar\partial$ also equals to $\tfrac{1}{4}$ the dimension of the full moduli space and must therefore coincide with the dimension of the vector space
of anti-equivariant Higgs fields with respect to $\bar\partial.$ Hence the dimension of the tangent space at the smooth point $(\bar\partial,0)$ of the moduli space of anti-equivariant Higgs bundles sums up to $\tfrac{1}{2}$ of the dimension of the full moduli space.
\end{proof}

Note that for $\Phi$ a Higgs field with spectral line bundle $L$, the Higgs field $-\Phi$ has spectral line bundle $\sigma^*L$. Thus, following similar steps as in the proof of Theorem \ref{Aabelian-variety},
 one can see that an anti-equivariant Higgs pair is determined (after choosing a square root of its determinant and after fixing a $\tau$-invariant base point $P_2$ in the   Prym variety)
by a point in the space
\begin{equation}\label{non-sym-prym}
\CPP^{\vee}:=\{L\in Jac(S)\mid \sigma^*L=L^*,\tau^*L=L^*\}\subset {\rm Prym}(S,\Sigma).
\end{equation}

\begin{remark}
When considering   involutions on the spectral data associated to Higgs bundles for the group $SL(2,\C)$, from \cite[Theorem 4.12]{thesis} one has that line bundles in ${\rm Prym}(S,\Sigma)$ of order two induce certain real Higgs bundles, namely $SL(2,\R)$-Higgs bundles. These are the line bundles which are both invariant and anti-invariant with respect to the involution defining the double  cover $S\rightarrow \Sigma$, and  the monodromy of the fibration has an explicit description in terms of spectral data   \cite{monodromy}.
\end{remark}

As seen before, the natural involution $\sigma$ of the spectral curve $S$ and the fixed point free involution $\tau$ commute on $S$. Hence,  $\tau\circ\sigma$ is an involution which  is fixed point free and $\sigma$ descends to an involution of the quotient Riemann surface $\tilde{S}:= S/(\tau\circ \sigma)$. Moreover, the quotient Riemann surface $\tilde  S/\sigma$ becomes $\Sigma_\tau$ in a natural way
yielding the following diagram:
\begin{equation}\label{diagMs}
\xymatrix{
&  S \ar[ld]_{\mod\sigma}\ar[d]  \ar[rd]^{\mod \tau\circ\sigma}  &\\
\Sigma\ar[rd]_{\mod\tau}   &S_\tau:=S/\tau \ar[d]   &  \tilde S\ar[ld]^{\mod\sigma} &\\
 & \Sigma_\tau =(S_\tau)/\sigma & \\
}
\end{equation}

\begin{proposition}\label{Babelian-variety}
The abelian variety $\CPP^{\vee}$ in \eqref{non-sym-prym} is given  by $$\CPP^{\vee}={\rm Prym}({\tilde S},\Sigma_\tau)/\Z_2,$$
where the generator of $\Z_2$ is the holomorphic 2-torsion line bundle $\tilde P_2$ on ${\tilde S}$ defining the cover $S\to   {\tilde S}$. 
\end{proposition}

\begin{proof} 
 Let  $\tilde P$ be a fixed base point in the affine Prym which is in $\CPP$.
Since $\sigma$ and $\tau$ commute,  then $
\CPP^{\vee}=\{L\in Jac(S)\mid \sigma^*L=L^*,(\sigma\circ\tau)^*L=L\}\subset {\rm Prym}(S,\Sigma),$
so the proposition follows  analogously to the proof of Theorem \ref{Aabelian-variety}.
\end{proof}

 \section{Hyperelliptic surfaces  of genus $3$}\label{equiv_data}\label{genus_3}
As seen in Theorem \ref{only1}, equivariant branes on Riemann surfaces $\Sigma$ of genus $3$ occur either with a fixed point free involution $\tau$ (as studied in Section \ref{fixer}), or with a hyperelliptic involution $\psi$ and a second involution $\rho$ which has $4$ fixed points. In what follows  both settings for compact Riemann surfaces of genus $3$ shall be considered.

\subsection{Equivariant branes through $\psi$ and $\rho$} In order to study the latter case appearing in Theorem \ref{only1} (II), consider    a compact Riemann surface $\Sigma$ of genus 3 and a  holomorphic quadratic differential $Q$ with simple zeros which is invariant under the involutions $\psi$  and  $\rho$ giving  a finite group action of $\Gamma$ as in Theorem \ref{only1} (II). 
In what follows, we shall study the intersection $\CPP$ of the mid-dimensional $(B,B,B)$ brane of equivariant Higgs bundles in Theorem  \ref{only1}  with the generic fibres of the $SL(2,\C)$ Hitchin fibration over $Q$,  where $Q$ is a generic (having simple zeros) element of $  H^{0}(\Sigma, K^2_\Sigma)$. 
 The double cover $ S$ defined in \eqref{hat} giving the spectral curve of Higgs fields $\Phi$ for which ${\rm det}(\Phi)=Q$ has genus $9$ and is defined as 
\begin{eqnarray}\pi\colon  S=\{\omega_p\in K_p\mid p\in\Sigma;\; -\omega_p^2=Q_p\}\to\Sigma.\end{eqnarray}
The cover inherits from ${\rm Tot}(K)$ a natural involution $\sigma:S\to S$  
with $\pi\circ\sigma=\pi,$ and the involutions $\psi$ and $\rho$ lift to commuting involutions (denoted by the same symbols) on $S$.
Note that neither of the involutions $\psi$ and $\rho$ on $  S$ has   fixed points since the points over the fixed points on $\Sigma$ are interchanged. As in the case of fixed point free involutions,  one may consider the quotient \begin{equation}\label{new-g3}
\tilde S=  S/(\Z_2\times\Z_2),\end{equation}
which is now a hyperelliptic surface of genus $3$, not necessarily the same as $\Sigma$. Its hyperelliptic involution is given by the induced action of $\sigma$, and
it branches over the   8 points on $\CP^1:$
six of them are the marked points of the   sphere as in Section \ref{CMC_2}, while the other two branch points are given by the two zeros
of the meromorphic quadratic differential on $\CP^1$ which pulls back to $Q.$

 As mentioned previously,  the intersection  $\CPP\subset {\rm Prym}(  S,\Sigma)$ of the mid-\linebreak dimensional $(B,B,B)$ brane of equivariant Higgs bundles with a generic fibre ${\rm Prym}(S,\Sigma)$ of the $SL(2,\C)$ Hitchin fibration parametrizes the ($\psi$- and $\rho$-) equivariant Higgs fields with determinant $Q$. In order to give a geometric description of this variety,  note the following:

\begin{proposition}\label{Afix}
The pullback (along $\tilde{S}\to S$) determines an immersion of ${\rm Jac}(\tilde S)$ into ${\rm Prym}(S,\Sigma)$ with finite kernel.
\end{proposition}


\begin{proof} 
The last part of the proposition is clear once one has the existence of a line bundle $M\to  S$ satisfying $M^2=K^*_\Sigma$ and which is invariant under $\psi$ and 
$\rho.$ Such a bundle needs to be the pull-back of $\mathcal O(-1)\to\CP^1$ by the fourfold covering $  S\to\CP^1$ given by factoring out $\sigma$ and $\psi.$
Given $L\in {\rm Jac}(\tilde S)$, one has that $L\otimes\sigma^*L$ is invariant under $\sigma$, and hence it is the pull-back of 
a holomorphic line bundle on $\CP^1=  \tilde S /\sigma.$ Moreover, since the cover is obtained through the abelian group $\Z_2\times \Z_2,$ its monodromy 
\begin{eqnarray}p\colon H_1(  \tilde S,\Z)\to \mathcal S_4\end{eqnarray}
is abelian, where   $\mathcal{S}_4$ is the 4-symmetric group. Furthermore, the action is of order two: for each $\gamma\in H_1( \tilde S,\Z)$ the composition is 
$p(\gamma)\circ p(\gamma)=\Id.$
In order to see when the pull-back of a holomorphic line bundle $L\in {\rm Jac}(\tilde S)$ becomes trivial  on $  S$, equip $L$ with its unique compatible unitary flat connection $\nabla^L.$ The bundle $L$ is in the kernel of the pull-back map if and only if the monodromy representation of 
the pull-back of $\nabla^L$ on $  S$ is trivial. This happens if and only if 
the monodromy of $\nabla^L$ along a closed curve $\gamma$ is $1$ when $p(\gamma)=\Id$, and $\pm1$ if $p(\gamma)\neq\Id$, and the set of such line bundles  is   finite.
\end{proof}

Following the notation from previous sections,   consider the two quotients  $\Sigma_\rho:=\Sigma/\rho$ and $S_\rho:=S/\rho$. Then, one has the following:
\begin{proposition}\label{tautwisted}
The branched  covers $ S/(\rho\circ\sigma)\to\Sigma_\rho$ and $S_\rho\to\Sigma_\rho$ differ by the $\Z_2$-bundle determining the unbranched cover $\Sigma\to\Sigma_\rho.$
\end{proposition}

\begin{proof} 
In order to prove the proposition, consider the concrete description of 2-fold covers in Section \ref{2cov}.
The surface $\Sigma$ is given by pairs
$(q,s_q)$ where $q\in \Sigma_\rho$ and $s_q$ is given by parallel transport along some curve (with fixed start point) and end point $q$
with respect to the unitary flat connection corresponding to the $\Z_2$-bundle $P_2\to\Sigma_\rho$. From that perspective
the spectral curve $S$ is given by triples
$(q, s_q,\omega_q)$
where $(q,s_q)$ are as above and $\omega_q\in (K_{\Sigma_\rho})_q$ satisfies $-\omega_q^2= (Q_\rho)_q.$
This identification holds as the pull-back of $K_{\Sigma_\rho}$ to $\Sigma$ is the canonical bundle $K_\Sigma.$ Then, the involutions $\sigma$ and $\rho$ act as
\begin{eqnarray}\sigma:  (q, s_q,\omega_q)&\mapsto& (q, s_q,-\omega_q),\\
 \rho:(q, s_q,\omega_q)&\mapsto& (q, -s_q,\omega_q).\end{eqnarray}
The spectral curve
$S_\rho$ is thus obtained  by identifying $(q,s_q,\omega_q)\sim (q,-s_q,\omega_q),$
and the curve
$S/(\rho\circ\sigma)$ by
identifying $(q,s_q,\omega_q)\sim (q,-s_q,-\omega_q).$ Taking the tensor product, i.e., $(q,s_q\otimes\omega_q)$,
which is well-defined on   $S/(\rho\circ\sigma),$ it follows that the branched cover
$ S/(\rho\circ\sigma)\to\Sigma_\rho$ is determined by   $Q_\rho\in H^0(\Sigma_\rho,K_{\Sigma_\rho}^2)$ and the holomorphic square root $P_2\otimes K_{\Sigma_\rho}$ of $K_{\Sigma_\rho}^2$ as required.
\end{proof}

The setting of Theorem \ref{theorem_3} (II) can be shown to be equivalent to the one of Theorem \ref{theorem_3} (I) for Riemann surfaces of genus $3$, reducing the study of the mid-dimensional equivariant $(B,B,B)$ branes to fixed point free actions on Riemann surfaces:

\begin{lemma}\label{0-4-8}
Let $\Sigma$ be a hyperelliptic Riemann surface of genus $3$ with hyperelliptic involution   $\psi$, equipped with an additional involution $\rho$ with $4$ fixed points. 
Then, $\tau=\rho\circ \psi$ is a fixed point free involution.
\end{lemma}
\begin{proof}
Note that since $\rho$ and $\phi$ commute, the map $\rho$ is an involution. Moreover, $\rho$ gives rise to an involution on the quotient $\Sigma/\phi=\CP^1,$ which must
have exactly two fixed points by Riemann-Hurwitz. The possible fixed points of $\rho$ on $\Sigma$ must map to the fixed points of $\rho$ on $\CP^1.$ 
Hence, there are only 4 possible fixed points, but these are already fixed points of $\rho$ by assumption, and therefore they are interchanged by $\tau=\phi\circ\rho.$
\end{proof}

\subsection{Equivariant branes through a fixed point free involution $\tau$}
From the previous analysis, the hyperelliptic involution $\phi$ on a compact Riemann surface $\Sigma$ of genus $g=3$ together with a fixed point free involution $\tau$ also induce an involution $\rho=\phi\circ\tau$ which has $4$ fixed points. Moreover, one can see that all genus $3$ Riemann surfaces with fixed point free actions must be hyperelliptic: 

\begin{proposition}\label{XXX}
Let $\Sigma$ be a Riemann surface of genus $3$ with a fixed point free involution $\tau$. Then, $\Sigma$ is hyperelliptic.
\end{proposition}
The proof is analogous to the proof of Lemma \ref{0-4-8}, 
and thus a hyperelliptic Riemann surface of genus $3$ with an additional involution with 4 fixed points is the same as Riemann surface of genus $3$ with a fixed point free involution. Hence,   the equivariant points in the Hitchin base can be described as follows:
\begin{proposition}
Let $\Sigma$ be a hyperelliptic Riemann surface of genus $3$ with an additional involution $\rho$ with 4 fixed points. Then, a holomorphic quadratic differential is invariant under the hyperelliptic involution $\phi$ and invariant under  $\rho$ if and only if
it is invariant under $\tau=\rho\circ\phi.$
\end{proposition}
\begin{proof} Without loss of generality, suppose that $\Sigma$ is given by the algebraic equation
\[y^2=(z-z_1)(z+z_1)(z-z_2)(z+z_2)(z-z_3)(z+z_3)(z-z_4)(z+z_4),\]
for 8 pairwise disjoint points $\pm z_1,..,\pm z_4\in\C\setminus\{0\}\subset\CP^1$,   and that\linebreak
$\phi: (y,z)\mapsto(-y,z),$ 
$\rho(y,z)= (y,-z)$ and  $\tau(y,z)= (-y,-z)$
A basis for 
$H^0(\Sigma, K_\Sigma^2)$ is then
$\{\frac{1}{y^2}(dz)^2, ~\frac{z}{y^2}(dz)^2, ~\frac{z^2}{y^2}(dz)^2,$ ~$ \frac{z^3}{y^2}(dz)^2,   $ ~$  \frac{z^4}{y^2}(dz)^2,  ~\frac{1}{y}(dz)^2\}.$
Therefore the space of $\tau$-invariant holomorphic quadratic differentials is span-ned by
$\{ \frac{1}{y^2}(dz)^2,\frac{z^2}{y^2}(dz)^2,~\frac{z^4}{y^2}(dz)^2\},$
and   this is exactly the space of $\phi$ and $\rho$ invariant holomorphic quadratic differentials.
\end{proof}
 
%
%

From the above result, if the gauge class of a flat connection is equivariant with respect to the hyperelliptic involution $\psi$ and with 
respect to $\rho$, it is   also equivariant with respect to $\tau.$ The converse is true as well for $\tau$-invariant connections (compare with Remark \ref{invariantsc}):
\begin{proposition}
Let $\nabla$ be a flat, irreducible, $\tau$-invariant $\SL(2,\C)$-con-nection on a Riemann surface $\Sigma$ of genus $3$, where $\tau$ is a fixed points free involution. Then, $\nabla$ is equivariant with respect to the hyperelliptic involution.
\end{proposition}
\begin{proof}
As $\nabla$ is invariant with respect to $\tau$ it is given by the pull-back of a flat $\SL(2,\C)$-connection on the genus $2$ surface $\Sigma_\tau=\Sigma/\tau.$ If $\nabla$ is irreducible, the corresponding connection on  $\Sigma_\tau$ is irreducible as well. Hence, by a result of 
\cite[Theorem 2.1]{HL14}, the connection is equivariant with respect to the hyperelliptic involution on the genus $2$ surface, and hence it is also equivariant with respect to the hyperelliptic involution on the genus 3 surface.
\end{proof}
Looking at the dimension of the various moduli spaces, it becomes clear that
$\tau$-invariant irreducible flat connections on $\Sigma$  are not invariant but only strictly equivariant with respect to
 the generators $\psi$ and $\rho$ of $\Gamma.$ 
\begin{remark}
Let $\Gamma$ be the group generated by $\psi$ and $\rho.$ The proposition above   states that (stable) strictly 
$\Gamma$-equivariant Higgs bundles are exactly (stable) $\tau$-invariant Higgs bundles. Hence, the intersection of the space of equivalence classes
of $\Gamma$-equivariant Higgs bundles with a regular fibre of the Hitchin map is given by Theorem \ref{Aabelian-variety}, see also Theorem \ref{ultimo} below.
\end{remark}
Note that irreducibility is not a necessary condition to be in the equivariant brane, as there exist flat reducible connections on $\Sigma$ which correspond to irreducible connections on the hyperelliptic genus 2 surface $\Sigma_\tau.$
Furthermore, since flat abelian connections on the quotient $\Sigma_\tau$ are equivariant with respect to the hyperelliptic involution, a non-trivial class in $H^1(\Sigma_\tau,\Z_2)$ is given by the choice of two points $(w_1,w_2)$ out of  the six Weierstrass points of
$\Sigma_\tau$.  Such a class is represented by a closed curve on $\CP^1\setminus\{w_1,..,w_6\}$ with even winding number around $w_3,..,w_6$ and odd winding number around $w_1$ and $w_2.$
 Without loss of generality,   suppose that $\Sigma_\tau$ is given by $y^2=(z-z_1)(z-z_2)...(z-z_6)$ and that the class in $H^1(\Sigma_\tau,\Z_2)$ labelling the double cover $\pi\colon \Sigma\to \Sigma_\tau$
 is  determined by $w_1=z_1$ and $w_2=z_2$, this is, is dual to the above homology class.  Then, the Riemann surface  $\Sigma$ has equation  $ u^2=(z_1-z_2)^2 (z_2-z_3)..(z_2-z_6)(w^2-\frac{z_1-z_3}{z_2-z_3})...(w^2-\frac{z_1-z_6}{z_2-z_6})$, the fixed point free involution $\tau$ is given by
$(u,w)\mapsto (-u,-w)$, 
and the covering map $\pi:\Sigma\rightarrow\Sigma_\tau$ is  given by
\begin{equation}\label{pi0infty}
\pi : (u,w)\mapsto (y,z)=\left(\frac{u w}{(w^2-1)^3},\frac{z_2 w^2-z_1}{w^2-1}\right).\end{equation}
  
In order to study the $SL(2,\C)$ Hitchin fibration on $\Sigma$, and the one induced on $\Sigma_\tau$, consider $Q$   a holomorphic quadratic differential on $\Sigma_\tau$ with simple zeros. 
After a Moebius transformation, and 
 up to constant scaling $Q=\frac{z(dz)^2}{y^2}$, and its pull-back   to $\Sigma$  is 
$\pi^*Q=4\frac{(z_1-z_2)^2(w^2-1)(z_2 w^2-z_1)(d w)^2}{u^2}.$
The quadratic differentials $Q$ and $\pi^*Q$ label $SL(2,\C)$-Higgs bundles, and thus define spectral curves  $S$ and $S_\tau$ which are double covers  of $\Sigma$ and $\Sigma_\tau$, respectively. Note that there is a natural
unbranched covering $S\to S/\tau= S_\tau.$
Thus, there is the following natural  commutative diagram, where as in previous sections, $\sigma$ is the natural involution switching the sheets of the 2-covers:
\begin{equation}\label{diagmMs3}
\xymatrix{
 S  \ar[rr]_{\mod\tau}\ar[d]^{\mod\sigma}&& S_\tau=S/\tau \ar[d]^{\mod\sigma}  \ar[rr]_{\mod\phi}&&\tilde\Sigma:=S_\tau/\phi\ar[d]^{\mod\sigma}\\
  \Sigma=S/\sigma \ar[rr]_{\mod\tau}&&\Sigma_\tau=\Sigma/\tau\ar[rr]_{\mod\phi}&&\CP^1
}
\end{equation}

Note that $\tilde\Sigma\to\CP^1$ branches over the points $0,\infty,z_1,..,z_6\in\CP^1.$
Denote the corresponding Weierstrass points of $\tilde\Sigma$ by the same symbols.
Moreover, one can show that the unbranched cover $S_\tau\to \tilde\Sigma$ corresponds to the pair of Weierstrass points $0,\infty\in\tilde\Sigma,$ 
and that the unbranched cover $S_\phi\to \tilde\Sigma$ corresponds to the pair of Weierstrass points $z_1,z_2\in\tilde\Sigma$, for $S_\phi:=S/\phi$.
This proves the first part of the following theorem:

\begin{theorem}\label{ultimo} The $(B,B,B)$-brane of $\Gamma$-equivariant $SL(2,\C)$-Higgs bundles intersects the generic fibres of the Hitchin fibration in an  abelian variety \begin{eqnarray}\mathcal{P}=Jac(\tilde\Sigma)/\Z_2\times \Z_2,\end{eqnarray} where $\tilde\Sigma$ is the hyperelliptic Riemann surface  of genus $3$ branched over $0,$ $\infty,  z_1,$ $\ldots ,z_6$
and $\Z_2\times\Z_2$ is isomorphic to the group generated by the $\Z_2$-bundles $L(0-\infty)\in{\rm Jac}(\tilde\Sigma) $ and $L(z_1-z_2)\in{\rm Jac}(\tilde\Sigma).$ \\
Analogously, one has that 
  \begin{eqnarray}\label{Pvee}\mathcal{P}^{\vee}= {\rm Jac}(E)\times {\rm Jac}(M)/\Z_2\times \Z_2,\end{eqnarray}  where
$E$ is the elliptic curve which branches over $0,\infty,z_1,z_2\in\CP^1$,
$M$ is the hyper-elliptic curve of genus 2 branched over $0,\infty, z_3,z_4,z_5,z_6\in\CP^1,$
and
$\Z_2\times\Z_2$ is isomorphic to the group generated by the $\Z_2$-bundles $L(z_1-z_2)\in {\rm Jac}(E)$ and  $L(0-\infty)\in{\rm Jac}(M).$

\end{theorem}
\begin{proof}

 In order to describe the points in the regular Hitchin fibres corresponding to $\tau$-anti-equivariant Higgs fields, recall from Porposition \ref{Babelian-variety} that these are given by the Prym variety of 
 $\tilde S=S/(\tau\circ\sigma)\to\Sigma_\tau$ modulo the $\Z_2$-line bundle $L_0\to\tilde S$  
 corresponding to the covering $S\to\tilde S.$  
In particular, this $\Z_2$-bundle is the unique non-trivial line bundle on $\tilde S$ which pulls back to the trivial bundle on $S$. In what follows we will show that $L_0$ is given by $L_0=L(z_1^+ +z_1^--z_2^+-z_2^-)\to \tilde S,$ where
$z_1^\pm$ and $z_2^\pm$ are the points in $\tilde S$ lying over $z_1,z_2\in\CP^1,$ respectively, in the commutative diagramm \eqref{diagRMs}. Indeed, this can be deduced from the fact   that   the maps $\tilde S\to \Sigma_\tau$ and  $S_\tau\to  \Sigma_\tau$ differ by the holomorphic $\Z_2$-bundle $L(z_1-z_2)$ on $\Sigma_\tau$, 
where  $L(z_1-z_2)\in Jac(\Sigma_\tau)$ is a holomorphic line bundle whose pull-back to $S$ is trivial.

Consider the Riemann surface of genus $2$ defined as
\[M:=S_\phi/(\tau\circ\sigma).\]  The natural action of $\sigma$ on the Riemann surface  $M$ has 6 fixed points lying over $0,\infty,z_3,z_4,z_5,z_6\in\CP^1$. Moreover, one can see that the induced map 
 $S_\phi\to S_\phi/(\tau\circ\sigma)=M$ branches over the 4 points lying over $z_1,z_2\in\CP^1$, leading to the following diagram:
\begin{equation}\label{diagRMs}
\xymatrix{
 & & & S \ar[llldd]^{\mod\sigma}\ar[dd]^{\mod\tau\circ\sigma}\ar[rrrdd]^{\mod\phi} & & & \\
& & & & & \\
      \Sigma\ar[rdd]^{\mod\tau}  & & & \tilde S\ar[lldd]^{\mod\sigma}\ar[rrdd]^{\mod\phi}\ar[dd]^{\mod\phi\circ\sigma}  & & & S_{\phi}\ar[ldd]^{\mod\tau\circ\sigma} & \\
 & & & & &\\
 &  \Sigma_\tau \ar[rrdd]^{\mod\phi} &  & E \ar[dd]^{\mod\sigma}& & M \ar[lldd]^{\mod\sigma} &\\
 & & & & &\\
 & &&  \CP^1   & &\\
}
\end{equation}

 From the above descriptions, it can be seen that the  Riemann surface $E:= \tilde S/(\phi\circ \sigma)$ is of genus 1. Furthermore, the induced action of $\sigma$ on $E$ gives a map $E\to\CP^1$ branched over $0,\infty,z_1,z_2\in\CP^1.$ Thus,
one obtains a 
 natural map (via pull-back composed with tensor product) given by  \begin{equation}\label{JacEMP}Jac(E)\times Jac(M)\to Prym(\tilde S,\Sigma_\tau).\end{equation}
 We will show next that   this map is surjective and its kernel is finite. 
 In fact, the first assertion follows from the second. 
 In order to compute the kernel of \eqref{JacEMP}
consider holomorphic line bundles $l_1\to E$ and $l_2\to M$ which have the property that the tensor product of their pull-backs to $\tilde S$  is trivial.

Note that all line bundles over $\tilde S$ which are contained in the pull-back of $Jac(E)$ are $\phi$-anti-invariant while those contained in the pull-back of $Jac(M)$ are $\phi$-invariant.
Hence, the intersection of the pull-backs of the two Jacobians is $0$-dimensional and consists 
of $\Z_2$-bundles $L_1$ on $\tilde S$.  The  line bundles  $l_1\to E$ and $l_2\to M$ 
which pull-back to $L_1\to \tilde S$
are also  $\Z_2$-bundles on $E$ and $M.$ Therefore it follows that either
$l_1=L(0-\infty)\to E$ and $l_2=L(0-\infty)\to M$ or both are trivial. 
 Noting 
that the $\Z_2$-bundle $L_0\to\tilde S$ is also given as the pull back of  $L(z_1-z_2)\to E,$ 
 the proof of \eqref{Pvee} (and of the theorem) follows.
\end{proof}

 \begin{remark}
 For $\Gamma$  the group generated by $\psi$ and $\rho$, note that $\Gamma$-anti-equivariant Higgs fields have not been defined. To see why this was not done, note that the space  $ \{L\in Prym(S,\Sigma)\mid \psi^*L=L^*,\rho^*L=L^* \}$ is 0-dimensional,   the space  $\{L\in Prym(S,\Sigma)\mid \psi^*L=L^*,\rho^*L=L\}$ is 2 dimensional, 
  and the space   $ \{L\in Prym(S,\Sigma)\mid \phi^*L=L,\rho^*L=L^* \}$
 is 1 dimensional. Hence, the only reasonable way to get a half-dimensional space would be to 
 consider the space of Higgs bundles for which there exists a $g\in\Gamma$ with respect to which the Higgs field is anti-invariant, and this is    the space of $\tau$-anti-equivariant Higgs bundles.
  \end{remark}

\section{Some remarks on equivariant branes and Langlands duality} \label{duality}

Langlands duality can be seen in terms of Higgs bundles as a duality between the fibres of the Hitchin fibrations for $\mathcal{M}_{G_\C}$ and $\mathcal{M}_{^LG_{\C}}$, for $^LG_\C$ the Langlands dual group of $G_\C$ (as was first seen in \cite{LPS_Tamas}).  As explained   \cite[Section 12]{Kap} under the duality, specifically homological mirror symmetry, there should be an equivalence of categories of branes on $\mathcal{M}_{G_\C}(\Sigma)$ and $\mathcal{M}_{^L {G_\C}}(\Sigma)$ under which the brane types $(B,B,B) \leftrightarrow (B,A,A)$ are exchanged.

Examples of branes and their proposed duals in the moduli spaces of Higgs bundles for low rank groups were presented in \cite{Kap}, and further studied in \cite{sergei}. Moreover, in the case of the $(B,A,A)$-branes coming from real forms $G$ of the complex lie group $G_\C$ there is a conjecture of what the (support of) the dual branes should look like \cite{slices} (see \cite{Ga} and \cite{classes} for support). It is thus natural to ask equivalent questions in the setting of the present research, about what the duality between branes should be in the case of the spaces constructed in this paper. 
 In the case of  $U(m,m)$-Higgs bundles, the duality was studied in \cite{classes} where through the spectral data description of \cite{umm} in terms of anti-invariant line bundles, the  proposed dual branes were constructed in terms of invariant ones, which agreed with the conjecture in \cite{slices}. 
 
 In \cite{classes} Hitchin proposed a hyper-holomorphic sheaf which together with the hyper-K\"ahler subspace of the moduli space of Higgs bundles would give the $(B,B,B)$-brane. It is interesting to note that given the similarities of the construction of the subspaces of Higgs bundles in terms of equivariant objects, the hyperholomorphic sheaf constructed in \cite{classes} for $U(m,m)$-Higgs bundles should give naturally a hyperholomorphic sheaf for the equivariant $(B,B,B)$-branes of the present paper.  
 
 Since from the work of Section \ref{equiv_data2} and   the previous propositions, the branes obtained are subspaces of abelian varieties too, following the lines of thought of the real case one may think that the
mirror of the equivariant $(B,B,B)$-brane is given by moduli space of Higgs bundles on $\Sigma$ which are $\psi$-equivariant and anti-equivariant with respect to $\tau.$
Here, anti-equivariant means that the corresponding holomorphic structure $\bar\partial$ is equivariant, this is, $\tau^*\bar\partial=\bar\partial.g$ 
and $\tau^*\Phi=-g^{-1}\Phi  g$ for a suitable gauge tranformation $g$.
Note that the determinant of an anti-equivariant Higgs field is an invariant holomorphic quadratic differential $Q$. As before, after the choice of a square root of $Q$,
anti-equivariant Higgs fields with determinant $Q$ are parametrized by the abelian variety $\CPP^{\vee}$ which consists of points in the ${\rm Prym}(S,\Sigma)$ which are invariant under $\psi$ and anti-invariant under $\tau.$





\end{document}